\documentclass[11pt]{article}

\usepackage[utf8]{inputenc} 
\usepackage{url}

\usepackage{amsmath,amsfonts,amssymb}
\usepackage{amsthm}
\numberwithin{equation}{section}
\usepackage{graphicx}
\usepackage{enumerate}
\usepackage{multirow,bigdelim}
\usepackage{mathtools}
\usepackage{soul} 
\usepackage{pstricks}
\usepackage[ruled, boxed,  linesnumbered, norelsize]{algorithm2e}
\usepackage[margin=3cm]{geometry}
\usepackage{array}
\usepackage[normalem]{ulem}
\usepackage{url}
\usepackage{hyperref}
\usepackage{enumerate}
\usepackage{cases}
\usepackage{mathtools}
\usepackage{todonotes}

\hypersetup{
  colorlinks,
  linkcolor={red!50!black},
  citecolor={blue!50!black},
  urlcolor={blue!80!black}
}

\renewcommand{\Re}{\mathbb{R}}

\DeclarePairedDelimiter\abs{\lvert}{\rvert}%

\theoremstyle{definition}

\theoremstyle{theorem}

\newtheorem{lemma}{Lemma}

\newtheorem{theorem}{Theorem}
\newtheorem{corollary}{Corollary}

\theoremstyle{remark}

\usepackage{pbox}
\usepackage{hyperref}
\usepackage{cleveref}
\expandafter\def\csname ver@etex.sty\endcsname{3000/12/31}

\usepackage{autonum}
\usepackage{microtype}
\newcommand{\cB}{\mathcal{B}}
\newcommand{\cX}{X}
\newcommand{\cY}{Y}
\DeclarePairedDelimiter\norm{\lVert}{\rVert}
\DeclarePairedDelimiter\prn{(}{)}
\DeclarePairedDelimiterX\inner[2]{\langle}{\rangle}{{#1},{#2}}
\DeclarePairedDelimiterX\Set[2]{\{}{\}}{\mspace{2mu}{#1}\;\delimsize|\;{#2}\mspace{2mu}}
\newcommand{\by}[2][]{\text{\pbox[c]{\textwidth}{(by \pbox[t]{\textwidth}{\phantom{}#2)#1}}}}
\crefname{equation}{}{}
\Crefname{equation}{Eq.}{Eqs.}
\crefname{proposition}{Proposition}{Propositions}
\crefname{theorem}{Theorem}{Theorems}
\crefname{lemma}{Lemma}{Lemmas}

\title{On the equivalence of a Hessian-free inequality and Lipschitz continuous Hessian}
\author{
Radu I. Bo\c{t}\thanks{
	Faculty of Mathematics, University of Vienna, Oskar-Morgenstern-Platz 1, 1090 Vienna, Austria.
	E-mail: \href{radu.bot@univie.ac.at}{radu.bot@univie.ac.at}
	}
\and
Minh N. Dao\thanks{
School of Science, RMIT University, Melbourne, VIC 3000, Australia. 
E-mail: \href{minh.dao@rmit.edu.au}{minh.dao@rmit.edu.au}
}
\and
Tianxiang Liu\thanks{
	School of Computing, Institute of Science Tokyo, Japan. E-mail: \href{liu@c.titech.ac.jp}{liu@c.titech.ac.jp}
}
\and
Bruno F. Louren\c{c}o\thanks{Department of Statistical Inference and Mathematics, Institute of Statistical Mathematics, Japan.
	E-mail: \href{bruno@ism.ac.jp}{bruno@ism.ac.jp}}
\and
Naoki Marumo\thanks{
Graduate School of Information Science and Technology, University of Tokyo, Tokyo, Japan.
E-mail: \href{marumo@mist.i.u-tokyo.ac.jp}{marumo@mist.i.u-tokyo.ac.jp}
}
}

\begin{document}
	\maketitle
  
\begin{abstract}
  It is known that if a twice differentiable function has a Lipschitz continuous Hessian, then its gradients satisfy a Jensen-type inequality.
  In particular, this inequality is Hessian-free in the sense that the Hessian does not actually appear in the inequality.
  In this paper, we show that the converse holds in a generalized setting: if a continuous function from a Hilbert space to a reflexive Banach space satisfies such an inequality, then it is Fr\'echet differentiable and its derivative is Lipschitz continuous.
  Our proof relies on the Baillon--Haddad theorem.  
\end{abstract}

\section{Introduction}

Carmon et al.~\cite{carmon2017convex} proposed a first-order method for minimizing nonconvex functions having Lipschitz continuous gradients and Hessians.
The idea is that even if the Hessian is not actually used, its mere Lipschitz continuity makes it possible to design faster first-order methods than what would be possible if the function only had Lipschitz continuous gradients.
In this vein, several algorithms have been proposed that offer theoretical or practical improvements (e.g., \cite{allen2018neon2,jin2018accelerated,li2023restarted,marumo2024parameter,marumo2024universal,xu2017neon+}).
Among these, Marumo and Takeda~\cite{marumo2024parameter} proposed a first-order method that does not require the Lipschitz constant of the Hessian as an input of the algorithm, unlike previous methods.
An important step in their analysis is establishing the following Jensen-type inequality.

\begin{lemma}[Hessian-free inequality {\cite[Lemma 3.1]{marumo2024parameter}}]\label{lem:hess_free}
  Let $f \colon \Re^d \to \Re$ be a twice differentiable function with $L$-Lipschitz continuous Hessian.
  Then, for any $x_1,\ldots, x_n \in \Re^d$ and $\lambda_1,\ldots, \lambda_n \geq 0$ such that $\sum _{i=1}^n \lambda_i = 1$, we have
  \begin{equation}\label{eq:hess_free}
    \norm*{
      \nabla f \prn*{\sum_{i=1}^n \lambda_i x_i }
      - \sum_{i=1}^n \lambda_i \nabla f(x_i)
    }
    \leq
    \frac{L}{2} \sum_{1 \leq i < j \leq n} \lambda_i \lambda_j\norm{x_i-x_j}^2.
  \end{equation}
\end{lemma} 	
The proof of \cref{lem:hess_free}, naturally, uses the fact that $f$ has a Lipschitz continuous Hessian. However, the resulting inequality \eqref{eq:hess_free} is free of Hessians and only involves the gradient of $f$.

A natural question then is the following:
{\begin{quote}
If $f$ is continuously differentiable and satisfies \eqref{eq:hess_free} for some constant $L >0$,  is $f$ necessarily twice differentiable? If so,  is its Hessian $L$-Lipschitz continuous?
\end{quote}}
It turns out that as a consequence of the Baillon--Haddad Theorem the answer is yes.
This is somewhat surprising given the fact that \eqref{eq:hess_free} does not involve Hessians at all.
In fact, we will prove the following \cref{thm:hess_free_converse}, a generalization of the answer to the Hilbert space setting.

In what follows, given Banach spaces $\cX$ and $\cY$ we denote the space of bounded linear operators between $\cX$ and $\cY$ by $\cB(\cX,\cY)$.
Let $X^*$ denote the dual space of $X$, i.e., $X^* \coloneqq \cB(\cX, \Re)$.
Also, for simplicity, throughout the paper we use the same notation $\norm{\cdot}$ to indicate the norms on different Banach spaces.
We recall that operator norm is defined for $T \in \cB(\cX,\cY)$ by $\norm*{T} \coloneqq \sup_{x \in \cX,\, \norm*{x} \leq 1} \norm*{T(x)}$.

\begin{theorem}\label{thm:hess_free_converse}
  Let $\cX$ and $\cY$ be real Hilbert and reflexive Banach spaces, respectively.
  Let $F\colon \cX \to \cY$ be a continuous function and let $L > 0$.
  Then the following are equivalent:
  \begin{enumerate}[$(i)$]
    \item\label{item:lip_F_derivative}
    $F$ is Fr\'echet differentiable on $\cX$ and its derivative $F' \colon \cX \to \cB(\cX,\cY)$ is $L$-Lipschitz continuous, i.e.,
    \begin{align}\label{eq:lip_F_derivative}
      \norm*{F'(x) - F'(y)}
      \leq
      L \norm*{x - y} \quad \forall x,y \in \cX,
    \end{align}
    where the norm on the left-hand side is the operator norm.
    \item\label{item:grad_free}
    For any $x_1,\ldots, x_n \in \cX$ and $\lambda_1,\ldots, \lambda_n \geq 0$ such that $\sum _{i=1}^n \lambda_i = 1$, the following holds:
    \begin{equation}\label{eq:grad_free}
      \norm*{
        F \prn*{\sum_{i=1}^n \lambda_i x_i} - \sum_{i=1}^n \lambda_i F(x_i)
      }    \leq
      \frac{L}{2} \sum_{1 \leq i < j \leq n} \lambda_i \lambda_j \norm{x_i-x_j}^2.
    \end{equation}
  \end{enumerate}	
\end{theorem}

The proof of $\eqref{item:lip_F_derivative} \implies \eqref{item:grad_free}$ is essentially the same as that of \cref{lem:hess_free}, since the proof in \cite{marumo2024parameter} does not rely on the assumptions $X = Y = \Re^d$ and $F = \nabla f$.
In this paper, we focus on the converse implication $\eqref{item:grad_free} \implies \eqref{item:lip_F_derivative}$.

Setting $\cX = \cY = \Re^d$ and $F = \nabla f$ establishes the converse of \cref{lem:hess_free}. 
Slightly more generally, the following result holds.
\begin{corollary}
  Let $X$ be a real Hilbert space and let $f \colon \cX \to \Re$ be a Fr\'echet differentiable function such that its gradient $\nabla f$ satisfies \eqref{eq:hess_free} for 
  every $x_1,\ldots, x_n \in \cX$ and $\lambda_1,\ldots, \lambda_n \geq 0$ such that $\sum _{i=1}^n \lambda_i = 1$.
  Then, $f$ is twice differentiable with $L$-Lipschitz continuous Hessian.
\end{corollary}

\section{Proof of the Theorem}
We start with a result that is contained in the enhanced version of the Baillon--Haddad Theorem described by Bauschke and Combettes in \cite{BC10}.
\begin{theorem}[A piece of the Baillon--Haddad Theorem {\cite[Theorem~2.1]{BC10}}]\label{theo:bh}
  Let $\cX$ be a real Hilbert space.
  Let $\beta > 0$ and suppose that $g \colon \cX \to \Re\cup\{+\infty\}$ is a proper,  convex and lower semicontinuous function.
  Then the following are equivalent:
  \begin{enumerate}[$(i)$]
    \item $g$ takes only real values,  it is Fr\'echet differentiable on $\cX$ and its gradient $\nabla g \colon \cX \to \cX$ is $\frac{1}{\beta}$-cocoercive\footnote{$G \colon \cX \to \cX$ is $\frac{1}{\beta}$-cocoercive if $\inner*{G(x) - G(y)}{x - y} \geq \frac{1}{\beta}\norm*{G(x) - G(y)}^2$ for all $x, y \in \cX$.};
    \item $\frac{\beta}{2} \norm*{\cdot}^2 - g$ is convex.
  \end{enumerate}	
\end{theorem}

We will prove \cref{thm:hess_free_converse} by reducing it to the case $\cY = \Re$ through the following lemma.

\begin{lemma}\label{lem:phi_derivative_lipschitz}
  Let $\cX$ and $\cY$ be real Hilbert and Banach spaces, respectively.
  Let $F\colon \cX \to \cY$ be a continuous function.
  Let $L > 0$ and suppose that \cref{eq:grad_free} holds for any $x_1,\ldots, x_n \in \cX$ and $\lambda_1,\ldots, \lambda_n \geq 0$ such that $\sum _{i=1}^n \lambda_i = 1$.
  For each $y^* \in \cY^*$, define the slice $\phi_{y^*} \colon \cX \to \Re$ by $\phi_{y^*} \coloneqq y^* \circ F$, i.e.,
  \begin{align}\label{eq:def_phi}
    \phi_{y^*}(x) = y^*(F(x)) \quad \forall x \in \cX.
  \end{align}
  Then, for each $y^* \in \cY^*$ with $\norm*{y^*} \leq 1$, the function $\phi_{y^*}$ is Fr\'echet differentiable everywhere and its gradient $\nabla \phi_{y^*} \colon \cX \to \cX$ is $L$-Lipschitz continuous.
\end{lemma}
\begin{proof}
  Inequality~\cref{eq:grad_free} gives for any $x, y \in \cX$ and $t \in [0, 1]$
  \begin{align}\label{eq:grad_free_special}
    \norm*{F(x + t (y - x)) - (1 - t) F(x) - t F(y)}
    \leq
    \frac{L}{2} t (1 - t) \norm{x - y}^2.
  \end{align}
  For every $x,y \in \cX$ and $t \in [0,1]$, we have
  \begin{alignat}{2}
    &
    \abs*{\phi_{y^*}(x+t(y-x)) - (1 - t) \phi_{y^*}(x) - t \phi_{y^*}(y)}\\
    ={}&
    \abs*{y^* \prn[\Big]{F(x+t(y-x)) - (1 - t) F(x) - t F(y) }}
    &\quad&\by{the definition \cref{eq:def_phi} of $\phi_{y^*}$}\\
    \leq{}&
    \norm*{F(x+t(y-x)) - (1 - t) F(x) - t F(y)}
    &\quad&\by{$\norm*{y^*} \leq 1$}\\
    \leq{}&
    \frac{L}{2} t (1 - t) \norm*{x - y}^2,
    &\quad&\by{\cref{eq:grad_free_special}}
  \end{alignat}
  which implies that 
  \begin{equation}\label{eq:component_free}
    -\frac{L}{2} t(1-t)\norm{x-y}^2
    \leq
    \phi_{y^*}(x+t(y-x)) - (1 - t) \phi_{y^*}(x) - t \phi_{y^*}(y)
    \leq
    \frac{L}{2} t(1-t)\norm{x-y}^2.
  \end{equation}
  The first and second inequalities in \eqref{eq:component_free}, together with the identity
  \begin{align}
    t(1-t)\norm{x-y}^2
    =
    (1-t)\norm{x}^2 + t\norm{y}^2 - \norm{x+t(y-x)}^2,
  \end{align}
  imply that $\frac{L}{2} \norm{\cdot}^2 - \phi_{y^*}$ and $\frac{L}{2} \norm{\cdot}^2 + \phi_{y^*}$ are convex, respectively.
  
  In particular, we have
  \[
  L \norm{\cdot}^2 - \underbrace{\prn*{\frac{L}{2} \norm{\cdot}^2 + \phi_{y^*}}}_{\text{convex}} = \underbrace{\frac{L}{2} \norm{\cdot}^2 - \phi_{y^*}}_{\text{convex}}.
  \]
  By \cref{theo:bh}, $\frac{L}{2} \norm{\cdot}^2 + \phi_{y^*}$ is Fr\'echet differentiable on $\cX$ and $\nabla (\frac{L}{2} \norm{\cdot}^2 + \phi_{y^*})$ is $\frac{1}{2L}$-cocoercive.
  In particular, $\phi_{y^*}$ is Fr\'echet differentiable on $X$.
  The cocoercivity of $\nabla (\frac{L}{2} \norm{\cdot}^2 + \phi_{y^*})$ implies that
  \begin{align}
    \inner*{L (x - y) + \nabla \phi_{y^*}(x) - \nabla \phi_{y^*}(y)}{x - y}
    \geq
    \frac{1}{2L} \norm*{L (x - y) + \nabla \phi_{y^*}(x) - \nabla \phi_{y^*}(y)}^2
    \quad
    \forall x,y \in \cX.
  \end{align}
  Expanding both sides and rearranging terms yields $\norm*{\nabla \phi_{y^*}(x) - \nabla \phi_{y^*}(y)}^2 \leq L^2 \norm*{x - y}^2$, which proves that $\nabla \phi_{y^*}$ is $L$-Lipschitz continuous.
\end{proof}

The biggest hurdle in proving \cref{thm:hess_free_converse} is establishing the Fr\'echet differentiability of $F$.
If $\cY$ were finite-dimensional, say, $\cY = \Re^d$, 
we would have $F = (F_1,\ldots, F_d)$ for certain functions $F_i \colon \cX \to \Re$.
These $F_i$ are, of course, the slices of $F$ defined by the usual unit vectors of $\Re^d$.
Then, \cref{lem:phi_derivative_lipschitz} would imply that all the $F_i$ have an $L$-Lipschitz derivative from which we would conclude the Fr\'echet differentiability of $F$ through elementary means.
The case where $\cY$ is infinite-dimensional is more delicate as it takes more effort to establish that $F$ is indeed Fr\'echet differentiable by analyzing its slices, as shown in the proof of the following lemma.

\begin{lemma}\label{lem:F_differentiability}
  Let $\cX$ and $\cY$ be real Banach spaces and suppose that $\cY$ is reflexive.
  Let $F\colon \cX \to \cY$ be a continuous function and let $L > 0$.
  For each $y^* \in \cY^*$ with $\norm*{y^*} \leq 1$, suppose that the slice $\phi_{y^*} \colon \cX \to \Re$ defined by \cref{eq:def_phi} is Fr\'echet differentiable and its derivative $\phi_{y^*}' \colon \cX \to \cX^*$ is $L$-Lipschitz continuous.
  Then, $F$ is Fr\'echet differentiable and its derivative $F' \colon \cX \to \cB(\cX, \cY)$ is $L$-Lipschitz continuous.
\end{lemma}

Note that, to make the statement more general, $\cX$ is not assumed to be a Hilbert space; rather, it is only assumed to be a Banach space, unlike in \cref{lem:phi_derivative_lipschitz}.
Therefore, we use the derivative $\phi_{y^*}' \colon \cX \to \cX^*$ in place of the gradient $\nabla \phi_{y^*} \colon \cX \to \cX$ in this lemma.

\begin{proof}[Proof of \cref{lem:F_differentiability}]
  By the Lipschitz continuity of $\phi_{y^*}'$, for every $x, h \in \cX$ and $y^* \in \cY^*$ with $\norm{y^*} \leq 1$, we have
  \begin{align}\label{eq:phi_smoothness}
    \abs[\Big]{
      \phi_{y^*}(x + h) - \phi_{y^*}(x) - \phi_{y^*}'(x) h
    }
    \leq
    \frac{L}{2} \norm*{h}^2.
  \end{align}

  The proof is divided into three parts:
  \begin{enumerate}
  \renewcommand\theenumi{(\arabic{enumi})} 
  \renewcommand{\labelenumi}{\rm (\arabic{enumi})}
    \item\label{bounded} 
    showing that for each $x \in \cX$, the map $y^* \mapsto \phi_{y^*}'(x)$ is a bounded linear map from $\cY^*$ to $\cX^*$,
    \item\label{derivative}
    constructing the Fr\'echet derivative of $F$ using the reflexivity of $Y$, and
    \item\label{Lipschitz}
    showing that the derivative is $L$-Lipschitz continuous.
  \end{enumerate}

  \paragraph{\ref{bounded}~Linearity and boundedness of $y^* \mapsto \phi_{y^*}'(x)$ for every $x \in X$.}
 Fix $x \in X$ arbitrarily.  The map $y^* \mapsto \phi_{y^*}(x)$ is linear by the definition \cref{eq:def_phi},  hence $y^* \mapsto \phi_{y^*}'(x)$ is a linear map between $\cY^*$ and $X^*$.
 
  Next, we show the boundedness. The continuity of $F$ at $x$ implies that there exists $\delta > 0$ such that for each $h \in \cX$ with $\norm{h} \leq \delta$ it holds
  \begin{align}\label{eq:F_continuity_bound}
    \norm*{F(x + h) - F(x)} \leq 1.
  \end{align}
Thus,  for each $y^* \in \cY^*$ with $\norm{y^*} \leq 1$, we have
  \begin{alignat}{2}
    \norm*{\phi_{y^*}'(x)}
    &=
    \frac{1}{\delta }\sup_{\norm*{h} \leq \delta} \abs*{\phi_{y^*}'(x) h}\\
    &\leq
    \frac{1}{\delta } \sup_{\norm*{h} \leq \delta} 
    \prn*{
      \abs*{\phi_{y^*}(x + h) - \phi_{y^*}(x)}
      + \frac{L}{2} \norm{h}^2
    }
    &\quad&\by{\cref{eq:phi_smoothness}}\\
    &\leq
    \frac{1}{\delta } \sup_{\norm*{h} \leq \delta} 
    \prn*{
      \norm*{F(x + h) - F(x)}
      + \frac{L}{2} \norm{h}^2
    }
    &\quad&\by{the definition \cref{eq:def_phi} of $\phi_{y^*}$\\and $\norm*{y^*} \leq 1$}\\
    &\leq
     \frac{1}{\delta } + \frac{L}{2} \delta,
    &\quad&\by{\cref{eq:F_continuity_bound}}
  \end{alignat}
  which proves the boundedness of $y^* \mapsto \phi_{y^*}'(x)$.

  \paragraph{\ref{derivative}~Constructing the Fr\'echet derivative of $F$.}
  We have shown that for each $x \in X$, the map $y^* \mapsto \phi_{y^*}'(x)$ is a bounded linear map between $Y^*$ and $X^*$.
  Therefore,  for each $x, h \in \cX$, the map $y^* \mapsto \phi_{y^*}'(x) h$ is an element of $Y^{**}$.
  The reflexivity of $\cY$ implies that for each $x, h \in \cX$, there exists a unique $f_x(h) \in \cY$ such that
  \begin{align}\label{eq:reflexivity}
    y^*(f_x(h)) = \phi_{y^*}'(x) h
    \quad \forall y^* \in \cY^*.
  \end{align}

  We will show that $f_x \colon \cX \to \cY$ (i.e., $h \mapsto f_x(h)$) is the Fr\'echet derivative of $F$ at $x$.
  Let $\cY^*_1 \coloneqq \Set{y^* \in \cY^*}{\norm*{y^*} \leq 1}$.
  The Hahn--Banach theorem implies that for any $y \in \cY$ we have
  \begin{align}\label{eq:hahn-banach}
    \norm*{y} = \sup_{y^* \in \cY^*_1} y^*(y),
  \end{align}
  see \cite[Corollary~6.7]{Co07}. Thus, for each $x, h \in \cX$, we have
  \begin{alignat}{2}
    \norm*{F(x + h) - F(x) - f_x(h)}
    ={}&
    \sup_{y^* \in \cY^*_1} y^* \prn[\Big]{F(x + h) - F(x) - f_x(h)}
    &\quad&\by{\cref{eq:hahn-banach}}\\
    ={}&
    \sup_{y^* \in \cY^*_1} \prn[\Big]{\phi_{y^*}(x + h) - \phi_{y^*}(x) - \phi_{y^*}'(x) h}
    &\quad&\by{\cref{eq:def_phi,eq:reflexivity}}\\
    \leq{}&
    \frac{L}{2} \norm*{h}^2
    = o(\norm{h}).
    &\quad&\by{\cref{eq:phi_smoothness}}
  \end{alignat}

  It remains to show that for each $x \in \cX$ the map $h \mapsto f_x(h)$ is linear and bounded.  Fix $x \in X$ arbitrarily.
  \Cref{eq:reflexivity} implies that for each $h_1, h_2 \in \cX$ and $y^* \in \cY^*$
  \begin{align}
    y^* \prn[\Big]{f_x(h_1 + h_2) - f_x(h_1) - f_x(h_2)}
    =
    \phi_{y^*}'(x) (h_1 + h_2)
    - \phi_{y^*}'(x) h_1
    - \phi_{y^*}'(x) h_2
    =
    0,
  \end{align}
  and hence $f_x(h_1 + h_2) = f_x(h_1) + f_x(h_2)$.
  Similarly we have $f_x(\alpha h) = \alpha f_x(h)$ for each $\alpha \in \Re$.
  Therefore, the map $h \mapsto f_x(h)$ is linear.
  We next show the boundedness.
  For each $h \in \cX$, we have
  \begin{alignat}{2}
    \norm*{f_x(h)}
    &=
    \sup_{y^* \in \cY^*_1} y^*(f_x(h))
    &\quad&\by{\cref{eq:hahn-banach}}\\
    &=
    \sup_{y^* \in \cY^*_1} \phi_{y^*}'(x) h
    &\quad&\by{\cref{eq:reflexivity}}\\
    &\leq
    \norm{h} \sup_{y^* \in \cY^*_1} \norm{\phi_{y^*}'(x)}.
  \end{alignat}
  Since the boundedness of $y^* \mapsto \phi_{y^*}'(x)$ was shown in the first part of this proof, the map $h \mapsto f_x(h)$ is bounded as well.
  
  \paragraph{\ref{Lipschitz}~Lipschitz continuity of the derivative.}
  The Lipschitz continuity of $x \mapsto f_x$ follows from the Lipschitz continuity of $\phi_{y^*}'$. Indeed, for each $x,y \in X$ it holds
  \begin{alignat}{2}
    \norm*{f_x - f_y}
    &=
    \sup_{\norm*{h} \leq 1,\ y^* \in \cY^*_1} y^* \prn*{f_x(h) - f_y(h)}
    &\quad&\by{\cref{eq:hahn-banach}}\\
    &=
    \sup_{\norm*{h} \leq 1,\ y^* \in \cY^*_1} \prn*{\phi_{y^*}'(x) - \phi_{y^*}'(y)} h
    &\quad&\by{\cref{eq:reflexivity}}\\
    &=
    \sup_{y^* \in \cY^*_1} \norm*{\phi_{y^*}'(x) - \phi_{y^*}'(y)}\\
    &\leq
    L \norm*{x - y},
    &\quad&\by{the Lipschitz continuity of $\phi_{y^*}'$}
  \end{alignat}
  which completes the proof.
\end{proof}

\Cref{thm:hess_free_converse} directly follows from \cref{lem:phi_derivative_lipschitz,lem:F_differentiability} as shown below.
\begin{proof}[Proof of $\eqref{item:grad_free} \implies \eqref{item:lip_F_derivative}$ in \cref{thm:hess_free_converse}]
  Let $F \colon \cX \to \cY$ be a continuous function satisfying \cref{eq:grad_free}.
  By \cref{lem:phi_derivative_lipschitz}, for each $y^* \in \cY^*$ with $\norm*{y^*} \leq 1$, the slice $\phi_{y^*}$ has an $L$-Lipschitz continuous gradient $\nabla \phi_{y^*}$.
  By \cref{lem:F_differentiability}, $F$ also has an $L$-Lipschitz continuous derivative $F'$.
\end{proof}

\section{Final remarks}
As usual, it may be interesting to see if \Cref{thm:hess_free_converse} can be further extended to more general settings. 
This would require some extension of Theorem~\ref{theo:bh}, which may be nontrivial, see \cite{WW22} for some discussion along these lines. 
Another source of difficulty is the fact that $\norm{\cdot}^2$ is not ensured to be differentiable in an arbitrary normed vector space.

{\small
	\section*{Acknowledgements}
	This work was initiated during the MATRIX workshop ``Splitting Algorithms -- Advances, Challenges, and Opportunities''. 
	The authors would like to thank MATRIX and the organizers of the workshop for the environment that allowed this project to flourish.
		
	This work was also partially supported by JSPS KAKENHI (24K23853) and JST CREST (JPMJCR24Q2).
}

\bibliographystyle{abbrvurl}
\bibliography{bib_plain}
\end{document}